\documentclass[a4paper,twoside,reqno]{amsart}

\usepackage[utf8]{inputenc}

\usepackage{authblk,orcidlink}
\usepackage{hyperref}
\hypersetup{urlcolor=blue, citecolor=red}
\usepackage{amsmath, amsthm, amssymb, xcolor, bbm, mathrsfs, graphicx, float, enumitem, orcidlink, subcaption}
\usepackage[english]{babel}
\usepackage{fancyhdr}
\usepackage[dvipsnames]{xcolor}

\newtheorem{theorem}{Theorem}
\newtheorem{lemma}[theorem]{Lemma}

\newcommand{\norm}[1]{\left\lVert#1\right\rVert}
\newcommand{\E}[1]{\mathbb{E}\left[#1\right]}

\newcommand{\Prob}{\mathbb{P}}

\title{Pairwise Attraction-Repulsion on Multilayer Social Networks}
\author{}
\date{}

\email{hsinlunl@math.nsysu.edu.tw}

\begin{document}\sloppy

\thispagestyle{firstpage}
\maketitle
\begin{center}
    Hsin-Lun Li \orcidlink{0000-0003-1497-1599}
    \centerline{$^1$Department of Applied Mathematics, National Sun Yat-sen University, Kaohsiung 804, Taiwan}
\end{center}
\medskip

\begin{abstract}
We introduce a probabilistic pairwise \emph{attraction--repulsion} model for
opinion dynamics on multilayer social networks, in which agents hold
layer-specific states and interact through random matchings that couple
multiple, time-varying layers. At each time step, interacting pairs update
their layer-specific states using layer-dependent, time-varying interaction
rates and a random sign (attractive or repulsive), and the resulting updates
are averaged across layers. This framework generalizes classical gossip and
Deffuant-type models while capturing heterogeneous cross-layer influences and
antagonistic interactions.

Under mild graph-theoretic and moment assumptions, we establish almost sure
global consensus. Specifically, when the expected net effect of interactions is
strictly attractive and random matchings ensure sufficient cross-layer
connectivity, all agents’ layer states converge almost surely to the global
average. We further identify a purely attractive regime in which consensus
holds even under intermittent connectivity and without any moment assumptions
on the initial states. Numerical experiments illustrate the dynamical regimes
predicted by the theory, including consensus, metastability, and polarization.
Together, these results provide a rigorous foundation for understanding how
multilayer structure, stochastic interactions, and mixed-sign influence shape
collective outcomes in social and engineered networked systems.
\end{abstract}

\section{Introduction}

Understanding how opinions evolve within societies is a central problem in
network science, statistical physics, and applied mathematics. Classical studies
have shown that the interplay between interaction rules and network topology can
produce rich emergent behavior, including consensus, polarization, clustering,
and persistent disagreement (see, e.g., the monographs by
Barrat et al.~\cite{barrat2008dynamical} and Latora et al.~\cite{latora2017complex}).
Fundamental opinion-dynamics models such as those of Deffuant
et al.~\cite{deffuant2000mixing} and Hegselmann~\cite{hegselmann2015opinion}
demonstrate how bounded confidence or averaging rules shape collective outcomes.
When the confidence threshold in the Deffuant model is removed, the dynamics
coincide with the randomized gossip algorithm of Boyd et
al.~\cite{boyd2006randomized}, in which randomly selected agent pairs perform
unconditional averaging. Despite their broad influence, these classical models
typically operate on a single network layer and rely exclusively on attractive
interactions.

Recent advances in complex-systems research emphasize that human interactions
are inherently \emph{multilayered}, reflecting diverse social contexts---family,
workplace, online platforms---and therefore cannot be faithfully represented by a
single graph. The development of multilayer network
theory~\cite{boccaletti2014structure,kivela2014multilayer} has revealed that
coupling multiple layers can generate phenomena that are irreducible to
single-layer dynamics. Opinion formation on such structures has been studied in
voter-like models~\cite{diakonova2016irreducibility}, majority-rule
dynamics~\cite{choi2019majority}, opinion--disease coevolution~\cite{velasquez2017interacting},
opinion-based centrality~\cite{reiffers2017opinion}, and spin-based interaction
systems with link overlap~\cite{kim2021link}. These works show that multiplexity
can lead to discontinuous transitions, nonmonotonic convergence times, and
strong inter-layer correlations.

Parallel lines of research have examined higher-order, interdependent, and
multi-topic interactions, which often produce qualitatively new dynamical
regimes~\cite{lambiotte2019networks}. Continuous-time and multi-dimensional
extensions of classical opinion models further highlight the need for flexible
frameworks capable of capturing heterogeneous and nonstandard interaction
rules~\cite{ye2020continuous,suo2008dynamics}. Meanwhile, modern
distributed-systems perspectives on asynchronous
updates~\cite{berenbrink2024asynchronous} and mean-field analyses of
non-consensus models~\cite{liu2024mean} underscore the importance of studying
settings where convergence is nontrivial or even absent.

A key limitation of most existing opinion models is their reliance on purely
\emph{attractive} interactions—agents move toward one another’s views. Yet
real social interactions frequently involve \emph{repulsion}, in which
individuals distance themselves from those with sharply diverging opinions,
reflecting negative influence, contrarian behavior, or social antagonism.
Although attraction--repulsion dynamics appear implicitly in some
bounded-confidence and spin-based models, a general analytical treatment on
multilayer, dynamically evolving networks remains largely unexplored.

\subsection*{Our contribution}

In this work, we propose a multilayer attraction--repulsion model on a
collection of time-varying social graphs, with one graph per layer. At each
time step, a random matching is selected, and each interacting pair updates
its layer-specific states according to a stochastic rule that allows for both
attraction and repulsion. Interaction rates are layer-dependent and may vary
over time, capturing heterogeneous and context-dependent social influences.

The main contributions of the paper are as follows:

\begin{enumerate}
   \item \textbf{Almost sure convergence to the global average.}
   We establish two complementary consensus results. First, under mixed
   attraction and repulsion, agents converge almost surely to the global
   average provided that the expected net effect of interactions is strictly
   attractive and suitable connectivity conditions are satisfied. This result
   requires only mild moment assumptions on the initial opinions. Second, when
   all interactions are purely attractive, almost sure convergence to the
   global average is guaranteed even for heavy-tailed initial states,
   including Cauchy distributions, without any moment assumptions. Both
   results accommodate stochastic, time-varying social graphs and cross-layer
   interactions.

   \item \textbf{A unified mathematical framework.}
   We introduce a single attraction--repulsion mechanism on multilayer
   networks that captures consensus, polarization, and metastable behavior
   within a unified mathematical framework. The model generalizes classical
   gossip and Deffuant-type dynamics and extends them to heterogeneous,
   antagonistic, and cross-layer-coupled settings.

   \item \textbf{Simulation results linking theory to emergent behavior.}
   We complement the theoretical analysis with numerical experiments on large
   synthetic networks, illustrating how the predicted dynamical regimes—
   consensus, metastability, and polarization—emerge across parameter space.
   These simulations highlight the interplay between multilayer structure and
   mixed-sign interactions in shaping collective outcomes.
\end{enumerate}

This work aims to bridge classical consensus theory—such as the foundational
results of Jadbabaie et al.~\cite{jadbabaie2003coordination} and the broader
families of averaging models~\cite{flache2017models}—with modern multilayer and
antagonistic interaction structures. By developing a flexible
attraction--repulsion mechanism on multiplex networks, we provide a unifying
and mathematically rigorous framework for understanding how network structure,
layer heterogeneity, and interaction polarity jointly shape opinion formation.

%substack;mask
\section*{Model Description}

Opinion dynamics in contemporary societies involve inherently multilayered and stochastic interactions, as individuals engage across diverse social contexts and adjust their behavior according to each setting. We formalize a \emph{pairwise attraction–repulsion mechanism} on multilayer social networks, where agents adjust their opinions according to stochastic attraction or repulsion, with layer-specific interaction rules.

Consider a set of \( n \) agents, each participating in \( m \) distinct layers. Let \( x_{ki}(t) \in \mathbb{R}^d \) denote the opinion of agent \( i \) in layer \( k \) at discrete time \( t \). Let \( G_{k,t} = ([n], E_{k,t}) \) denote the simple, undirected social graph in layer \( k \) at time $t$, with vertex set \( [n] = \{1, \ldots, n\} \) and edge set
\[
E_{k,t} = \{(i,j) : \text{agents } i \text{ and } j \text{ are socially connected at time } t\}.
\]
A \emph{matching} of a graph \( G = (V,E) \) is a set of pairwise disjoint edges in \( E \). At each time step \( t \), a matching \( U_t \) of the complete graph of order \( n \) is selected, where \( \{U_t\}_{t \ge 0} \) are independent and identically distributed random variables with support
\[
S \subset \{\text{all matchings in the complete graph of order } n\}.
\]

Define
\[
L_{(i,j),t} = \{ k : (i,j) \in E_{k,t} \}, \quad (i,j) \in [n]^2,
\]
the set of layers containing the social edge \( (i,j) \) at time~\( t \), and
\[
M_t = \{ (i,j) \in U_t : L_{(i,j),t} \neq \emptyset \}.
\]

For each \( (i,j) \in M_t \), agents \( i \) and \( j \) update their opinions toward or away from each other at layer-dependent rates \( \mu_{\ell,(i,j)}(t) \in (0,1/2] \) for all \( \ell \in L_{(i,j),t} \). The updated opinions are averaged across all relevant layers:
\begin{equation}
\begin{aligned}
x_{\ell i}(t+1) &= \frac{1}{|L_{(i,j),t}|} \sum_{k \in L_{(i,j),t}} \Big[ x_{ki}(t) + r_{k,(i,j)}(t) \big( x_{kj}(t) - x_{ki}(t) \big) \Big], \\
x_{\ell j}(t+1) &= \frac{1}{|L_{(i,j),t}|} \sum_{k \in L_{(i,j),t}} \Big[ x_{kj}(t) + r_{k,(i,j)}(t) \big( x_{ki}(t) - x_{kj}(t) \big) \Big],
\end{aligned}
\label{Eq:Model}
\end{equation}
where
\[
r_{k,(i,j)}(t) =
\begin{cases}
\mu_{k,(i,j)}(t), & \text{with probability } \theta_{k,(i,j)}(t), \\
-\mu_{k,(i,j)}(t), & \text{with probability } 1 - \theta_{k,(i,j)}(t).
\end{cases}
\]

Let \( (M_t, F_t) \) be the undirected graph with vertex set \( M_t \) and edge set
\[
F_t = \{ (e_1, e_2) \in M_t^2 : e_1 \neq e_2 \text{ and } L_{e_1,t} \cap L_{e_2,t} \neq \emptyset \}.
\]
We say that \( M_t \) forms a \emph{connected collection of all layers} if \( \bigcup_{e \in M_t} L_{e,t} = [m] \) and \( (M_t, F_t) \) is connected. A graph is \( \delta \)-\emph{trivial} if the distance between any two vertices is at most~\( \delta \).

For simplicity, we introduce the following notation. Given a family of edges
\(X_{i,t} \subset E_{i,t}\), define
\[
\mathscr{A}(X_{i,t})
= \inf_{t \ge 0}\;
\min_{\, i \in [m],\, (p,q) \in \left(\bigcup_{a \in S} a\right) \cap X_{i,t}}
\mathbb{E}\!\left[\mu_{i,(p,q)}(t)\,\bigl(2\theta_{i,(p,q)}(t) - 1 - \mu_{i,(p,q)}(t)\bigr)\right].
\]
Moreover, define
\[
\mu(X_{i,t})=
\inf_{t \ge 0}\;
\min_{i \in [m],\, (p,q) \in \left(\bigcup_{a \in S} a\right) \cap X_{i,t}}
\mu_{i,(p,q)}(t),
\]
and
\[
\theta(X_{i,t})=
\inf_{t \ge 0}\;
\min_{i \in [m],\, (p,q) \in \left(\bigcup_{a \in S} a\right) \cap X_{i,t}}
\theta_{i,(p,q)}(t).
\]

We call an edge $(p,q)$ \emph{admissible} if $(p,q)\in\bigcup_{a\in S} a$, that is,
if it appears in at least one matching in the support of the matching process.
The quantity $\mathscr{A}(X_{i,t})$ represents the \emph{uniform worst-case
expected net attraction} generated by interactions along admissible edges in
$X_{i,t}$, taken uniformly over all layers and all times. In particular,
$\mathscr{A}(X_{i,t})>0$ ensures a uniform positive lower bound on the expected
net effect of interactions along admissible edges in $X_{i,t}$, across all
layers and times.

Similarly, $\mu(X_{i,t})$ and $\theta(X_{i,t})$ capture uniform lower bounds on
the interaction strength and the attraction probability, respectively, taken
over all admissible edges, layers, and times. Together, these quantities provide
compact sufficient conditions that rule out vanishing interaction strength and
degenerately repulsive behavior in the long run.

Compared to classical consensus or Deffuant models, our framework generalizes:
(i) layer-dependent and time-varying interaction rates \(\mu_{k,(i,j)}(t)\), 
(ii) stochastic attraction/repulsion, 
(iii) interactions determined by random matchings, and 
(iv) cross-layer averaging of opinions. This formulation enables the study of rich multilayer phenomena, including partial consensus, clustering, oscillatory dynamics, and metastable repulsive states.

\section*{Main Results}

We now present the main theoretical results for consensus in the pairwise
attraction--repulsion model on multilayer social networks. These theorems
establish conditions under which agents converge to a common opinion, even in
the presence of time-varying interaction strengths, stochastic attraction and
repulsion, and multiple interaction layers.  

\begin{theorem}[Almost Sure Consensus under Mixed Interactions]
\label{Thm:consensus}
Let \( W_t = \sum_{i \in [m],\, j \in [n]} \|x_{ij}(t)\|^2 \). Assume that
\begin{itemize}
  \item the initial opinions have finite second moments, i.e., 
  \( \mathbb{E}[W_{0}] < \infty \),

  \item the union of all matchings in the support contains the edges of the
  spanning forests of all layers:
  \[
  \bigcup_{a \in S} a \supset 
  \bigcup_{i \in [m],\, t \ge 0} E\big(\mathcal{F}(G_{i,t})\big),
  \]

  \item the time-varying attraction--repulsion parameters satisfy
  \[ 
  \mathscr{A}(E_{i,t})\ge 0\quad\hbox{and}\quad\mathscr{A}\left(E\big(\mathcal{F}(G_{i,t})\big)\right)>0,
  \]

  \item for all sufficiently large \(t\), all social graphs \(G_{k,t}\), \(k \in
  [m]\), are connected, and \(M_t\) forms a connected collection of all layers.
\end{itemize}

Then the opinions of all agents converge almost surely to the global average:
\[
x_{ij}(t) \to \frac{1}{nm} \sum_{p \in [m],\, q \in [n]} x_{pq}(0)
\quad \text{as } t \to \infty,
\]
for all \( i \in [m] \) and \( j \in [n] \).
\end{theorem}

\begin{theorem}[Consensus under Pure Attraction]
\label{Thm:consensus under attraction only}
Assume that
\begin{itemize}
  \item the union of all matchings in the support contains the edges of the
  spanning forests of all layers:
  \[
  \bigcup_{a \in S} a \supset 
  \bigcup_{i \in [m],\, t \ge 0} E\big(\mathcal{F}(G_{i,t})\big),
  \]

  \item all interactions are attractive, i.e., \(\theta(E_{i,t})=1\), 
  
  \item the time-varying interaction strengths satisfy \( \mu\left(E\big(\mathcal{F}(G_{i,t})\big)\right)>0,\)

  \item for infinitely many \(t\), all social graphs \(G_{k,t}\), \(k \in [m]\),
  are connected, and \(M_t\) forms a connected collection of all layers.
\end{itemize}

Then the opinions of all agents converge almost surely to the global average:
\[
x_{ij}(t) \to \frac{1}{nm} \sum_{p \in [m],\, q \in [n]} x_{pq}(0)
\quad \text{as } t \to \infty,
\]
for all \(i \in [m]\) and \(j \in [n]\).
\end{theorem}

These theorems provide a rigorous foundation for understanding consensus
formation in realistic multilayer social networks with stochastic and
time-varying interactions. In Theorem~\ref{Thm:consensus}, the assumption that the
initial opinions have finite second moments ensures that agents’ starting
opinions are not excessively dispersed. The requirement that the union of
matchings in the support contains the spanning forests of all layers, together
with the connectivity of the social graphs and the formation of connected
collections of matchings, guarantees that information can propagate both within
and across layers.

The interaction assumptions further distinguish the two results.
In Theorem~\ref{Thm:consensus}, interactions may involve both attraction and
repulsion; however, the expected net effect along admissible edges is required
to be strictly attractive. This condition ensures that, on average, pairwise
interactions drive agents toward one another, leading to almost sure consensus.
In contrast, Theorem~\ref{Thm:consensus under attraction only} assumes purely
attractive interactions, which allows the connectivity requirements to be
relaxed: consensus is achieved even when social graphs are only intermittently
connected, provided such connectivity occurs infinitely often.

Taken together, these results demonstrate that global consensus can emerge in
multilayer systems even under heterogeneous, stochastic, and time-varying
interactions. Theorem~\ref{Thm:consensus} highlights robustness to the presence
of mixed attraction and repulsion, showing that occasional repulsive
interactions do not prevent consensus as long as the average influence remains
attractive. In contrast, Theorem~\ref{Thm:consensus under attraction only}
shows that purely attractive interactions suffice to guarantee consensus under
weaker connectivity requirements. This framework therefore extends classical
consensus and Deffuant-type models by incorporating multilayer structure,
probabilistic repulsion, and random pairwise interactions, offering practically
meaningful insights into opinion dynamics in modern, complex social networks.

\section{Properties of the model}

We establish key mathematical properties of the multilayer
attraction--repulsion model that underpin the main results presented
earlier. We begin by recalling basic convexity and inequality results that
facilitate the analysis of opinion updates. We then construct a suitable
supermartingale capturing the evolution of opinion dispersion across all
agents and layers, which allows us to rigorously control the effects of
stochastic, time-varying, and potentially repulsive interactions. Building
on this framework, we prove lemmas that ensure the eventual triviality of
all connected components in the social graphs, both under mixed-sign
interactions and under purely attractive dynamics. These results collectively
provide the technical foundation for the almost sure convergence of all
agents’ opinions to the global average.

\begin{lemma}\label{Lemma:convex function}
The function \( f(x) = \|x\|^2 \), where \( x \in \mathbb{R}^d \), is convex.
\end{lemma}

\begin{proof}
For all \( t \in (0,1) \) and \( x, y \in \mathbb{R}^d \), observe that
\begin{align*}
    &\| (1 - t) x + t y \|^2 \le (1 - t) \| x \|^2 + t \| y \|^2 \\
    &\iff (1 - t)^2 \| x \|^2 + t^2 \| y \|^2 + 2 t (1 - t) \langle x, y \rangle \le (1 - t) \| x \|^2 + t \| y \|^2 \\
    &\iff (1 - t) t \| x \|^2 + t (1 - t) \| y \|^2 - 2 t (1 - t) \langle x, y \rangle \ge 0 \\
    &\iff \| x \|^2 + \| y \|^2 - 2 \langle x, y \rangle \ge 0 \\
    &\iff \| x - y \|^2 \ge 0.
\end{align*}
Therefore, the function \( f(x) = \| x \|^2 \) on \( \mathbb{R}^d \) is convex. This completes the proof.
\end{proof}

\begin{lemma}[Jensen's inequality]\label{Lemma:Jensen's inequality}
Assume that a function \( f(x) \) defined on \( D \subset \mathbb{R}^d \) is convex. Let \( x_i \in \mathbb{R}^d \) and \( t_i > 0 \) for all \( i \in [n] \), with \( \sum_{i \in [n]} t_i = 1 \). Then,
\[
f\left( \sum_{i \in [n]} t_i x_i \right) \le \sum_{i \in [n]} t_i f(x_i).
\]
\end{lemma}

\begin{proof}
Since \( f \) is convex on \( D \), the inequality holds for \( n = 2 \). For \( n > 2 \),
\begin{align*}
f\left( \sum_{i \in [n]} t_i x_i \right)
&= f\left( (1 - t_n) \sum_{i \in [n-1]} \frac{t_i}{1 - t_n} x_i + t_n x_n \right) \\
&\le (1 - t_n) f\left( \sum_{i \in [n-1]} \frac{t_i}{1 - t_n} x_i \right) + t_n f(x_n) \\
&\le (1 - t_n) \sum_{i \in [n-1]} \frac{t_i}{1 - t_n} f(x_i) + t_n f(x_n) \\
&= \sum_{i \in [n]} t_i f(x_i).
\end{align*}
By induction, the statement holds for all \( n \ge 2 \). It is clear that the inequality holds for \( n = 1 \), so the result is true for all \( n \ge 1 \).
\end{proof}

\begin{lemma}\label{Lemma:supermartingale}
Let \( W_t = \sum_{i \in [m],\, j \in [n]} \|x_{ij}(t)\|^2 \). Then,
\[
W_t - W_{t+1} \ge 2 \sum_{(p,q) \in M_t}
\sum_{i \in L_{(p,q),t}} r_{i,(p,q)}^2(t) \left( r_{i,(p,q)}^{-1}(t) - 1 \right) \| x_{iq}(t) - x_{ip}(t) \|^2,
\]
and
\begin{align*}
\E{ W_t - W_{t+1} \mid M_t }
&\ge
2 \sum_{(p,q) \in M_t}
\sum_{i \in L_{(p,q),t}}
\E{ \mu_{i,(p,q)}(t) \left( 2 \theta_{i,(p,q)}(t) - 1 - \mu_{i,(p,q)}(t) \right) } \\
&\hspace{3.5cm} \times \E{ \| x_{iq}(t) - x_{ip}(t) \|^2 }.
\end{align*}
If \(\E{W_{0}} < \infty\) and \(\mathscr{A}(E_{i,t})\ge 0\),
then \( (W_t)_{t \ge 0} \) is a supermartingale.
\end{lemma}

\begin{proof}
Let \( x_{ij} = x_{ij}(t) \), \( r_{i,(p,q)} = r_{i,(p,q)}(t) \), and 
\( a_{i,(p,q)} = r_{i,(p,q)} \big( x_{iq} - x_{ip} \big) \) for all 
\( i \in [m] \), \( j \in [n] \), and \( (p,q) \in [n]^2 \).
It is clear that \( W_t - W_{t+1} = 0 \) if \( M_t = \emptyset \).
For \( M_t \neq \emptyset \),
\begin{align}
W_t - W_{t+1}
&= \sum_{(p,q) \in M_t} 
\bigg[
\sum_{i \in L_{(p,q),t}} \big( \| x_{ip}  \|^2 + \| x_{iq}  \|^2 \big) \notag \\
&\hspace{2cm}
- |L_{(p,q),t}| \left\|
\frac{
1
}{|L_{(p,q),t}|}\sum_{i \in L_{(p,q),t}} (x_{ip} + a_{i,(p,q)})
\right\|^2 \notag \\
&\hspace{2cm}
- |L_{(p,q),t}| \left\|
\frac{
1
}{|L_{(p,q),t}|}\sum_{i \in L_{(p,q),t}} (x_{iq} - a_{i,(p,q)})
\right\|^2
\bigg]. \label{Eq:difference}
\end{align}
By Lemmas~\ref{Lemma:convex function} and~\ref{Lemma:Jensen's inequality},
\begin{align*}
&|L_{(p,q),t}|
\left\|
\frac{
1
}{|L_{(p,q),t}|}\sum_{i \in L_{(p,q),t}} (x_{ip} + a_{i,(p,q)})
\right\|^2
\le 
\sum_{i \in L_{(p,q),t}} \| x_{ip} + a_{i,(p,q)}  \|^2,
\end{align*}
and similarly,
\begin{align*}
|L_{(p,q),t}|
\left\|
\frac{
1
}{|L_{(p,q),t}|}\sum_{i \in L_{(p,q),t}} (x_{iq} - a_{i,(p,q)})
\right\|^2
\le 
\sum_{i \in L_{(p,q),t}} \| x_{iq} - a_{i,(p,q)}  \|^2.
\end{align*}
Therefore, from \eqref{Eq:difference},
\begin{align*}
W_t - W_{t+1}
&\ge 
\sum_{(p,q) \in M_t}
\sum_{i \in L_{(p,q),t}}
\big[
\| x_{ip}  \|^2 - \| x_{ip}  + a_{i,(p,q)} \|^2 \\
&\hspace{3.5cm}
+ \| x_{iq}  \|^2 - \| x_{iq}  - a_{i,(p,q)} \|^2
\big] \\
&= 
\sum_{(p,q) \in M_t}
\sum_{i \in L_{(p,q),t}}
\big(
-2 \| a_{i,(p,q)} \|^2
- 2 \langle x_{ip} , a_{i,(p,q)} \rangle \\
&\hspace{3.5cm}
+ 2 \langle x_{iq} , a_{i,(p,q)} \rangle
\big) \\
&= 
\sum_{(p,q) \in M_t}
\sum_{i \in L_{(p,q),t}}
\big(
-2 \| a_{i,(p,q)} \|^2
+ 2 \langle x_{iq} - x_{ip}, a_{i,(p,q)} \rangle
\big)\\
&=
2\sum_{(p,q) \in M_t}
\sum_{i \in L_{(p,q),t}}(r_{i,(p,q)}^{-1}-1)\| a_{i,(p,q)} \|^2.
\end{align*}
Taking expectations, we obtain
\begin{align}
\mathbb{E}\!\left[ W_t - W_{t+1} \mid M_t \right]
&\ge
2 \sum_{(p,q) \in M_t}
\sum_{i \in L_{(p,q),t}}
\mathbb{E}\big[ r_{i,(p,q)}^2 \big( r_{i,(p,q)}^{-1} - 1 \big) \big]
\mathbb{E}\big[ \| x_{iq}-x_{ip} \|^2 \big]
 \label{Eq:Independence} \\
&\hspace{-1cm}=
2 \sum_{(p,q) \in M_t}
\sum_{i \in L_{(p,q),t}}
\mathbb{E}\Big[
\mu_{i,(p,q)}
\big( 2\theta_{i,(p,q)} - 1 - \mu_{i,(p,q)} \big)
\Big]
\mathbb{E}\big[ \| x_{iq} - x_{ip} \|^2 \big]. \notag
\end{align}
Here,~\eqref{Eq:Independence} holds due to the independence of $r_{i,(p,q)}$ and $x_{iq}, x_{ip}$.  
Moreover, from the assumption and the fact that
\[
M_t \subset \bigcup_{i \in [m]} \left( E_{i,t} \cap \bigcup_{a \in S} a \right),
\]
we have
\[
\mathbb{E}\big[ W_t - W_{t+1} \mid M_t \big] \ge 0
\quad \text{for all } t \ge 0.
\]
This completes the proof.
\end{proof}

Since \( (W_t)_{t \ge 0} \) is a nonnegative supermartingale when $\mathscr{A}(E_{i,t})\ge 0$,
the martingale convergence theorem implies that \( W_t \) converges almost surely to some random variable \( W \) with finite expectation, as stated in Lemma~\ref{Lemma:W_t converges}. 

\begin{lemma}\label{Lemma:W_t converges}
    \( W_t \) converges almost surely to some random variable \( W \) with finite expectation as \( t \to \infty \) whenever $\E{W_{0}}<\infty$ and $\mathscr{A}(E_{i,t})\ge 0$.
\end{lemma}

Since a spanning forest is itself a graph, \( E(\mathcal{F}(G_{i,t})) \) denotes the edge set of the spanning forest of the graph \( G_{i,t} \).

\begin{lemma}\label{Lemma:DeltaTriviality}
Assume that \( \mathbb{E}[W_{0}] < \infty \),
\[
\bigcup_{a \in S} a \supset \bigcup_{i \in [m],\, t \ge 0} E\bigl(\mathcal{F}(G_{i,t})\bigr), 
\qquad 
\mathscr{A}\!\left(E\bigl(\mathcal{F}(G_{i,t})\bigr)\right) > 0,
\qquad 
\text{and} \qquad 
\mathscr{A}(E_{i,t}) \ge 0.
\]
Then, for all \( \epsilon > 0 \), all components of \( G_{k,t} \), where \( k \in [m] \), are \( \epsilon \)-trivial almost surely after some time step.
\end{lemma}

\begin{proof}
Let \( (\Omega, \mathscr{F}, \Prob) \) be a probability space. Suppose, for contradiction, that the statement is false. Then there exist \( k \in [m] \) and \( \epsilon > 0 \) such that some component of \( G_{k,t} \) is \( \epsilon \)-nontrivial infinitely often on an event \( F \in \mathscr{F} \) with \( \Prob(F) > 0 \). Since all social graphs are finite and
\[
\bigcup_{a \in S} a \supset \bigcup_{i \in [m],\, t \ge 0} E\big(\mathcal{F}(G_{i,t})\big),
\]
it follows from the triangle inequality and the second Borel--Cantelli lemma that there exist
\( \delta = \epsilon/n > 0 \) and a social edge \( (p,q) \in M_t \bigcap E\bigl(\mathcal{F}(G_{k,t})\bigr)\) such that
\( \|x_{k p}(t) - x_{k q}(t)\| > \delta \)
infinitely often on the event \( F \).

By Lemma~\ref{Lemma:supermartingale}, we have
\begin{align*}
\mathbb{E}\bigl[\,W_t - W_{t+1} \mid M_t\,\bigr]
\;\ge\;
2\,\mathscr{A}\!\left(E\bigl(\mathcal{F}(G_{k,t})\bigr)\right)\,
\delta^{2}\,\Prob(F).
\end{align*}
infinitely often. On the other hand, Lemma~\ref{Lemma:W_t converges} and the dominated convergence theorem imply that
\[
\mathbb{E}[W_t - W_{t+1} \mid M_t] \to 0
\qquad \text{as } t \to \infty,
\]
which contradicts the inequality above.
\end{proof}

\begin{lemma}\label{Lemma:Delta-trivialityUnderAttractionOnly}
Assume that
\[
\bigcup_{a \in S} a \supset \bigcup_{i \in [m],\, t \ge 0}
E\big(\mathcal{F}(G_{i,t})\big),\quad
\mu\Big(E\big(\mathcal{F}(G_{i,t})\big)\Big) > 0,\quad
\text{and}\quad \theta(E_{i,t}) = 1.
\]
Then, for all \( \epsilon > 0 \), all components of \( G_{k,t} \), where \( k \in [m] \), are \( \epsilon \)-trivial almost surely after some time step.
\end{lemma}

\begin{proof}
Let \( (\Omega, \mathscr{F}, \Prob) \) be a probability space. Suppose, for contradiction, that the statement is false. Then there exist \( k \in [m] \) and \( \epsilon > 0 \) such that some component of \( G_{k,t} \) is \( \epsilon \)-nontrivial infinitely often on an event \( F \in \mathscr{F} \) with \( \Prob(F) > 0 \). Since all social graphs are finite and
\[
\bigcup_{a \in S} a \supset \bigcup_{i \in [m],\, t \ge 0}
E\big(\mathcal{F}(G_{i,t})\big),
\]
it follows from the triangle inequality and the second Borel--Cantelli lemma that there exist
\( \delta = \epsilon/n > 0 \) and a social edge \( (p,q) \in M_t\bigcap E\bigl(\mathcal{F}(G_{k,t})\bigr) \) such that
\( \|x_{kp}(t) - x_{kq}(t)\| > \delta \)
infinitely often on the event $F$.

By Lemma~\ref{Lemma:supermartingale}, we have
\[
W_t - W_{t+1} \;\ge\; 2 \left(\mu\Big(E\big(\mathcal{F}(G_{i,t})\big)\Big) \right)^2 \, \delta^2
\]
infinitely often on the event $F$. On the other hand, \( (W_t)_{t \ge 0} \) is nonnegative and nonincreasing when
\(\theta(E_{i,t}) = 1.\) Thus \( W_t \) converges almost surely to some random variable \( W \) as \( t \to \infty \). Consequently,
\[
W_t - W_{t+1} \to 0
\qquad \text{as } t \to \infty,
\]
contradicting the inequality above.
\end{proof}

\begin{lemma}\label{Lemma: delta-trivial of all opinions}
Let \( \theta(E_{i,t}) = 1 \).  
For any \( \delta > 0 \), if
\[
\max_{i_1, i_2 \in [m],\, j_1, j_2 \in [n]} \norm{x_{i_1 j_1}(t) - x_{i_2 j_2}(t)} \leq \delta,
\]
then
\[
\max_{i_1, i_2 \in [m],\, j_1, j_2 \in [n]} \norm{x_{i_1 j_1}(t+1) - x_{i_2 j_2}(t+1)} \leq \delta
\]
for all \( t \ge 0 \).
\end{lemma}

\begin{proof} 
Under the assumption, for all \( t \ge 0 \), \( i \in [m] \), and \( j \in [n] \), each \( x_{ij}(t+1) \) lies in the convex hull of the set \( \{ x_{pq}(t) : p \in [m],\, q \in [n] \} \).  
Therefore,
\begin{align*}
&\max_{i_1,i_2 \in [m],\, j_1,j_2 \in [n]} \norm{x_{i_1 j_1}(t+1) - x_{i_2 j_2}(t+1)} \\
&\hspace{3cm}\le \max_{i_1,i_2 \in [m],\, j_1,j_2 \in [n]} \norm{x_{i_1 j_1}(t) - x_{i_2 j_2}(t)} \le \delta.
\end{align*}
\end{proof}

\begin{proof}[\bf Proof of Theorem~\ref{Thm:consensus}]
We claim that
\[
\lim_{t \to \infty} \max_{i_1, i_2 \in [m],\, j_1, j_2 \in [n]} \norm{x_{i_1 j_1}(t) - x_{i_2 j_2}(t)} = 0.
\]
By Lemma~\ref{Lemma:DeltaTriviality}, for every \(\delta > 0\), there exists a time \(s \ge 0\) such that all components in all social graphs are \(\delta / (n + 2)\)-trivial for all \(t \ge s\). From the assumptions, there is a time \(s_1 \ge s\) such that, for all \(t \ge s_1\), all social graphs are connected, with \(M_t\) forming a connected collection of all layers.

Let \((p_1, q_1)(p_2, q_2) \dots (p_{r-1}, q_{r-1})(p_r, q_r)\) be the shortest path in the graph \((M_t, F_t)\) such that \(i_1 \in L_{(p_1, q_1), t}\) and \(i_2 \in L_{(p_r, q_r), t}\) for all \(t \ge s_1\) and \(i_1, i_2 \in [m]\). Let \(k_i \in L_{(p_i, q_i), t} \cap L_{(p_{i+1}, q_{i+1}), t}\) for all \(i \in [r-1]\) and \(t \ge s_1\). Then \(r \le n + 1\).

Define \(x_{ij}^\star = x_{ij}(t+1)\) for all \(i \in [m]\) and \(j \in [n]\). By the triangle inequality,
\begin{align*}
    \norm{x_{i_1 j_1}^\star - x_{i_2 j_2}^\star} &\le 
    \norm{x_{i_1 j_1}^\star - x_{k_1 p_1}^\star} + \sum_{i \in [r-2]} \norm{x_{k_i p_i}^\star - x_{k_{i+1} p_{i+1}}^\star} \\
    &\hspace{0.5cm} + \norm{x_{k_{r-1} p_{r-1}}^\star - x_{i_2 p_r}^\star} + \norm{x_{i_2 p_r}^\star - x_{i_2 j_2}^\star} \\
    &\le (r + 1) \delta / (n + 2) \le \delta
\end{align*}
for all \(t \ge s_1\), \(i_1, i_2 \in [m]\), and \(j_1, j_2 \in [n]\), where \(x_{k_{i+1} p_{i+1}}^\star = x_{k_i p_{i+1}}^\star\) for all \(i \in [r-2]\), \(x_{k_1 p_1}^\star = x_{i_1 p_1}^\star\), and \(x_{i_2 p_r}^\star = x_{k_{r-1} p_r}^\star\). This proves the claim.

Observe that the total opinion of all individuals across all layers is conserved over time, and
\[
\frac{1}{nm} \sum_{i \in [m],\, j \in [n]} x_{ij}(t) \quad \text{belongs to the convex hull of} \quad \{ x_{ij}(t) \}_{i \in [m],\, j \in [n]}.
\]
Therefore,
\begin{align*}
   & \limsup_{t \to \infty} \norm{x_{ij}(t) - \frac{1}{nm} \sum_{p \in [m],\, q \in [n]} x_{pq}(0)} = \limsup_{t \to \infty} \norm{x_{ij}(t) - \frac{1}{nm} \sum_{p \in [m],\, q \in [n]} x_{pq}(t)} \\
   &\quad \le \limsup_{t \to \infty} \max_{i_1, i_2 \in [m],\, j_1, j_2 \in [n]} \norm{x_{i_1 j_1}(t) - x_{i_2 j_2}(t)} = 0
\end{align*}
for all \(i \in [m]\) and \(j \in [n]\). Therefore, \(x_{ij}(t)\) converges to
\[
\frac{1}{nm} \sum_{p \in [m],\, q \in [n]} x_{pq}(0) \quad \text{as} \quad t \to \infty
\]
for all \(i \in [m]\) and \(j \in [n]\).
\end{proof}

\begin{proof}[\bf Proof of Theorem~\ref{Thm:consensus under attraction only}]
    By Lemma~\ref{Lemma:Delta-trivialityUnderAttractionOnly}, for all \(\delta > 0\), there exists a time \(s \ge 0\) such that all components in all social graphs are \(\delta / (n + 2)\)-trivial for all \(t \ge s\). By the assumptions, all social graphs are connected at some time \(s_1 + 1 > s\), with \(M_{s_1}\) forming a connected collection of all layers.

    Let \((p_1, q_1)(p_2, q_2) \dots (p_{r-1}, q_{r-1})(p_r, q_r)\) be the shortest path in the graph \((M_{s_1}, F_{s_1})\) such that \(i_1 \in L_{(p_1, q_1), s_1}\) and \(i_2 \in L_{(p_r, q_r), s_1}\) for all \(i_1, i_2 \in [m]\). Let \(k_i \in L_{(p_i, q_i), s_1} \cap L_{(p_{i+1}, q_{i+1}), s_1}\) for all \(i \in [r-1]\). Then \(r \le n + 1\).

    Define \(x_{ij}^\star = x_{ij}(s_1 + 1)\) for all \(i \in [m]\) and \(j \in [n]\). By the triangle inequality,
    \begin{align*}
        \norm{x_{i_1 j_1}^\star - x_{i_2 j_2}^\star} &\le 
        \norm{x_{i_1 j_1}^\star - x_{k_1 p_1}^\star} + \sum_{i \in [r-2]} \norm{x_{k_i p_i}^\star - x_{k_{i+1} p_{i+1}}^\star} \\
        &\hspace{0.5cm} + \norm{x_{k_{r-1} p_{r-1}}^\star - x_{i_2 p_r}^\star} + \norm{x_{i_2 p_r}^\star - x_{i_2 j_2}^\star} \\
        &\le (r + 1) \delta / (n + 2) \le \delta
    \end{align*}
    for all \(i_1, i_2 \in [m]\) and \(j_1, j_2 \in [n]\), where \(x_{k_{i+1} p_{i+1}}^\star = x_{k_i p_{i+1}}^\star\) for all \(i \in [r-2]\), \(x_{k_1 p_1}^\star = x_{i_1 p_1}^\star\), and \(x_{i_2 p_r}^\star = x_{k_{r-1} p_r}^\star\). Together with Lemma~\ref{Lemma: delta-trivial of all opinions}, this yields
    \[
    \lim_{t \to \infty} \max_{i_1, i_2 \in [m],\, j_1, j_2 \in [n]} \norm{x_{i_1 j_1}(t) - x_{i_2 j_2}(t)} = 0.
    \]

    Observe that the total opinion of all individuals across all layers is conserved over time, and
    \[
    \frac{1}{nm} \sum_{i \in [m],\, j \in [n]} x_{ij}(t) \quad \text{belongs to the convex hull of} \quad \{x_{ij}(t)\}_{i \in [m],\, j \in [n]}.
    \]
    Therefore,
    \begin{align*}
       & \limsup_{t \to \infty} \norm{x_{ij}(t) - \frac{1}{nm} \sum_{p \in [m],\, q \in [n]} x_{pq}(0)} = \limsup_{t \to \infty} \norm{x_{ij}(t) - \frac{1}{nm} \sum_{p \in [m],\, q \in [n]} x_{pq}(t)} \\
       &\quad \le \limsup_{t \to \infty} \max_{i_1, i_2 \in [m],\, j_1, j_2 \in [n]} \norm{x_{i_1 j_1}(t) - x_{i_2 j_2}(t)} = 0
    \end{align*}
    for all \(i \in [m]\) and \(j \in [n]\). Therefore, \(x_{ij}(t)\) converges to
    \[
    \frac{1}{nm} \sum_{p \in [m],\, q \in [n]} x_{pq}(0) \quad \text{as} \quad t \to \infty
    \]
    for all \(i \in [m]\) and \(j \in [n]\).
\end{proof}

\section{Simulations}

We illustrate the theoretical results established in
Theorems~\ref{Thm:consensus} and~\ref{Thm:consensus under attraction only}
through numerical simulations. All simulations were implemented in
\textsf{R}, with random seeds fixed to ensure reproducibility.

\subsection{Common simulation setup}

We consider a system of \( n = 100 \) agents interacting over \( m = 5 \)
distinct layers, simulated over a time horizon of \( T = 1000 \).
At each time step \( t \), a social graph \( G_{k,t} \) is independently
generated for each layer \( k \in [m] \) as an Erd\H{o}s--R\'enyi graph with
edge probability \( p = 0.05 \).
To guarantee connectivity in every layer, we force the inclusion of the
chain
\[
1\text{--}2\text{--}\cdots\text{--}n,
\]
so that each layer contains a spanning path.

At each time step, a matching \(U_t\) is sampled, and the effective interaction set
\(M_t\) is obtained by retaining only those matched edges that appear in at least one layer. For each interacting pair, opinions are updated according to the multilayer
attraction--repulsion dynamics defined in~\eqref{Eq:Model}, with
layer-dependent interaction rates.

For visualization, we plot the trajectories of all opinions
\( x_{ki}(t) \) over a finite time window and superimpose the global average
\[
\frac{1}{nm}\sum_{k=1}^{m}\sum_{i=1}^{n} x_{ki}(0),
\]
which is conserved by the dynamics.

\subsection{Simulation of Theorem~\ref{Thm:consensus}}

To reflect the assumptions of Theorem~\ref{Thm:consensus}, the support of the
matching distribution is restricted so that each matching contains at least
one chain edge \( (i,i+1) \). This ensures that \(M_t\) forms a connected collection of all layers at every time step.

Interaction parameters are sampled as follows.
For chain edges \( (i,i+1) \),
\[
\mu \sim \mathrm{Unif}(0.1, 0.5),
\qquad
\theta \sim \mathrm{Unif}\!\left( \frac{1.1 + \mu}{2},\, 1 \right),
\]
while for all other edges,
\[
\mu \sim \mathrm{Unif}(0, 0.5),
\qquad
\theta \sim \mathrm{Unif}\!\left( \frac{1 + \mu}{2},\, 1 \right).
\]
Initial opinions are drawn independently from the uniform distribution on
\( (0,1) \).

Figure~\ref{fig:thm1} displays representative opinion trajectories up to
time \( t = 800 \).
Despite the presence of repulsive interactions, all opinions converge almost
surely to the global average, in agreement with the conclusion of
Theorem~\ref{Thm:consensus}.

\subsection{Simulation of Theorem~\ref{Thm:consensus under attraction only}}

For Theorem~\ref{Thm:consensus under attraction only}, we consider a purely
attractive regime by eliminating repulsive interactions and allowing
matchings that do not necessarily contain chain edges.
Matchings are sampled uniformly at random from the complete graph.

Interaction strengths remain edge-dependent:
for chain edges \( (i,i+1) \),
\[
\mu \sim \mathrm{Unif}(0.1, 0.5),
\]
and for non-chain edges,
\[
\mu \sim \mathrm{Unif}(0, 0.5).
\]
Initial opinions are sampled independently from a Cauchy distribution,
highlighting that convergence does not rely on finite second moments.

Figure~\ref{fig:thm2} shows opinion trajectories up to time \( t = 400 \).
Even under heavy-tailed initial conditions and only intermittent
connectivity, the system reaches consensus, corroborating the robustness of
Theorem~\ref{Thm:consensus under attraction only}.

\begin{figure}[htbp]
  \centering
  \begin{subfigure}[b]{0.49\textwidth}
    \centering
    \includegraphics[width=\textwidth]{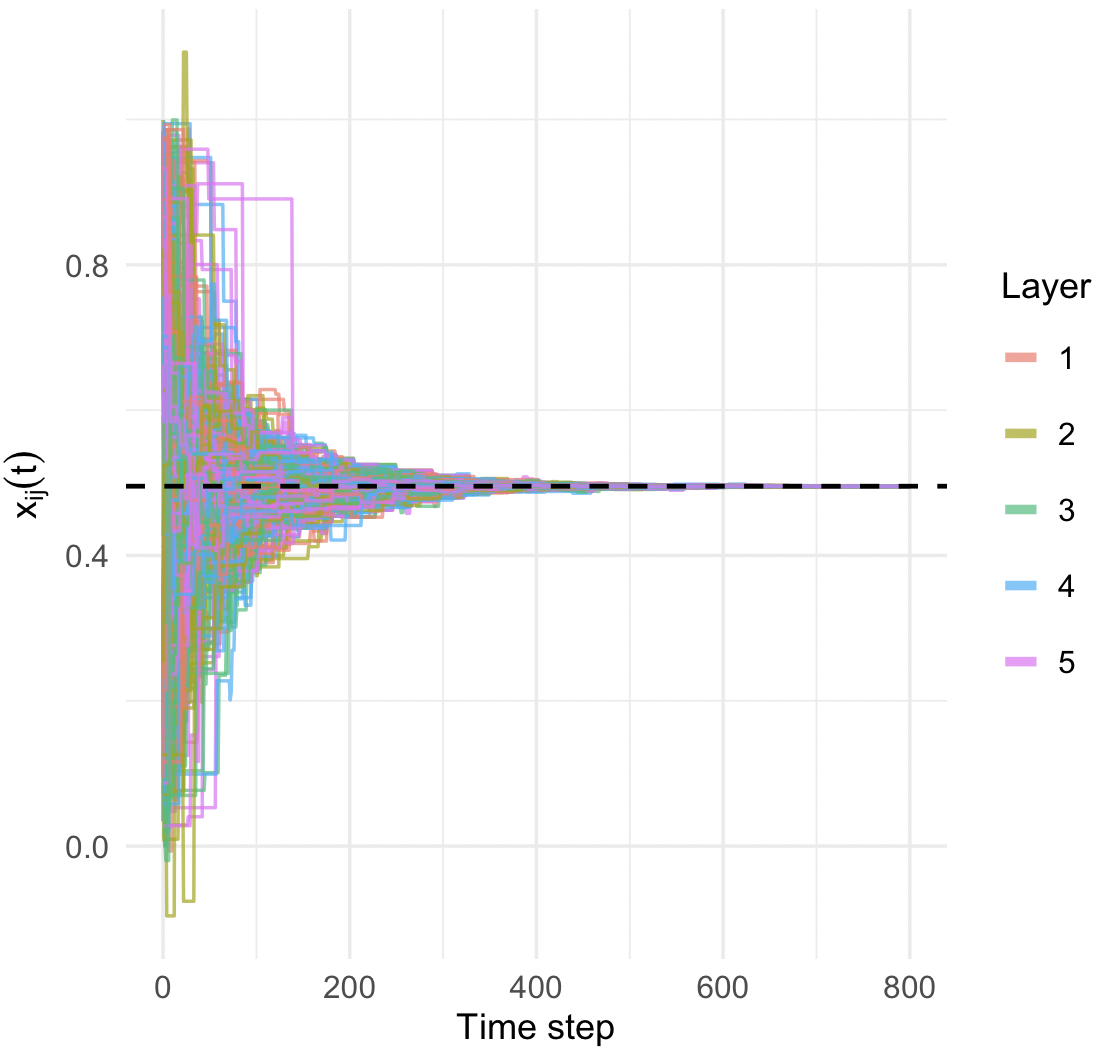}
    \caption{Attraction--repulsion dynamics}
    \label{fig:thm1}
  \end{subfigure}
  \hfill
  \begin{subfigure}[b]{0.49\textwidth}
    \centering
    \includegraphics[width=\textwidth]{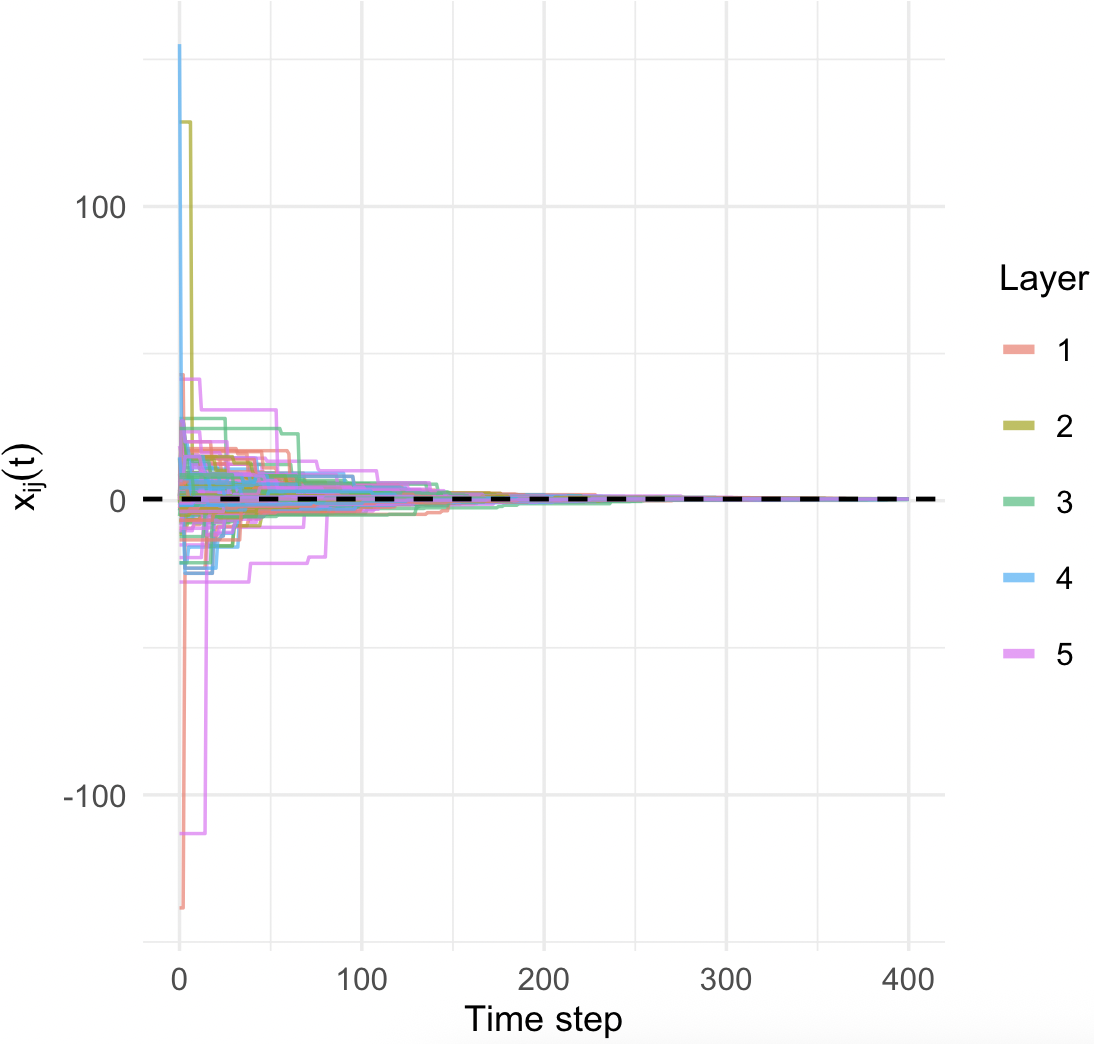}
    \caption{Pure-attraction dynamics}
    \label{fig:thm2}
  \end{subfigure}
  \caption{Evolution of multilayer opinions. In both panels, each colored trajectory corresponds to an individual agent--layer opinion \(x_{ki}(t)\), and the dashed horizontal line indicates the conserved global average \( \frac{1}{nm}\sum_{k=1}^{m}\sum_{i=1}^{n} x_{ki}(0) \).}
  \label{fig:thm}
\end{figure}

\section{Discussion and Further Study}

This work develops a rigorous framework for opinion dynamics on multilayer social networks with stochastic, time-varying, and possibly antagonistic interactions. By establishing almost sure consensus under both mixed attraction--repulsion and purely attractive regimes, the results clarify how global agreement can emerge despite heterogeneity across layers, randomness in interactions, and intermittent connectivity. The analysis highlights the roles of admissible interactions, connectivity across layers, and the average effect of social influence in shaping long-run collective behavior.

Several directions for further study naturally arise from this framework. One important extension is to relax the matching-based interaction structure to allow for simultaneous multi-agent interactions or weighted subgraph updates, which may better capture real-world communication patterns. Another direction is to investigate phase transitions between consensus, polarization, and persistent disagreement by characterizing sharp thresholds in the balance between attraction and repulsion. Incorporating endogenous network evolution, layer-dependent noise, or strategic agent behavior may further enrich the model and improve its descriptive power. Finally, extending the analysis to continuous-time dynamics or to settings with external information sources would broaden the applicability of the framework to a wider class of social and information systems.

\section{Statements and Declarations}

\subsection{Competing Interests}
The author is partially funded by an NSTC grant.

\subsection{Data Availability}
No associated data were used in this study.

%\nocite{*}
%\bibliographystyle{unsrt}  % or unsrt, alpha, etc.
%\bibliography{refs}

\end{document}